\documentclass[a4paper]{amsart}
\usepackage[latin1]{inputenc}
\usepackage{amssymb}
\usepackage{amsmath}
\usepackage{mathrsfs}
\usepackage{eufrak}
\usepackage{amsthm}
\usepackage{amsfonts}
\usepackage{textcomp}
\usepackage{graphicx}
\usepackage[pdftex]{color}
\usepackage{paralist}
\usepackage[shortlabels]{enumitem}
\usepackage{hyperref}
\usepackage{comment}
\usepackage[arrow, matrix, curve]{xy}
\usepackage{tikz-cd}

\newtheorem*{corollary*}{Corollary}
\newtheorem*{conjecture*}{Conjecture}
\newtheorem*{example*}{Example}
\newtheorem*{theorem*}{Theorem}
\newtheorem*{proposition*}{Proposition}

\newtheorem{theorem}{Theorem}[section]

\newtheorem{proposition}[theorem]{Proposition}

\newtheorem*{claim*}{Claim}

\newtheorem*{question}{Question}

\theoremstyle{definition}

\theoremstyle{remark}

\numberwithin{equation}{section}

\makeatletter
\renewcommand*\env@matrix[1][\
arraystretch]{%
  \edef\arraystretch{#1}%
  \hskip -\arraycolsep
  \let\@ifnextchar\new@ifnextchar
  \array{*\c@MaxMatrixCols c}}
\makeatother

\newcommand{\Ext}{\operatorname{Ext}}

\newcommand{\Tr}{\operatorname{Tr}}
\newcommand{\Hom}{\operatorname{Hom}}

\newcommand{\rad}{\operatorname{\mathrm{rad}}}

\begin{document}

\title{On a question of Auslander and Bridger on 2-reflexive modules}
\date{\today}



\author[Marczinzik]{Ren\'{e} Marczinzik}
\address[Marczinzik]{Mathematical Institute of the University of Bonn, Endenicher Allee 60, 53115 Bonn, Germany}
\email{marczire@math.uni-bonn.de}

\begin{abstract}
We answer a question raised by Auslander and Bridger by showing that not every 2-reflexive module is reflexive.
\end{abstract}

\maketitle

Let $R$ be a two-sided noetherian ring and $M$ a finitely generated $R$-module.
$M$ is called \emph{reflexive} if the canonical evaluation map $\operatorname{ev}_M: M \rightarrow M^{**}$ given by $\operatorname{ev}_M(m)(\psi)=\psi(m)$ with $m \in M, \psi \in M^{*}$ is an isomorphism, where $N^*:=\Hom_R(N,R)$ for an $R$-module $N$. The study of reflexive modules is a classical topic in ring theory and algebraic geometry, we refer for example to \cite{E,EK,CM}.
Let $P_1 \xrightarrow{f} P_0 \rightarrow M \rightarrow 0$ denote a projective presentation of an $R$-module $M$, then the \emph{Auslander-Bridger transpose} $\Tr M$ is defined as the cokernel of the map $\Hom_R(f,R)$ and we denote by $D_k(M)$ the module $\Omega^k \Tr(M)$.
Following Auslander and Bridger \cite[Page 3]{AB}, a module $M$ is said to be $k$-reflexive for $k \geq 0$ when the natural map $M \rightarrow D_k^2(M)$ induces an isomorphism $\Ext_R^1(D_k^2(M),-) \cong \Ext_R^1(M,-)$, we refer to \cite[Chapter 2]{AB} for the precise definition of this natural map and more information on $k$-reflexive modules. 
The notion of $k$-reflexive modules can for example be used to characterise Gorenstein commutative rings, see for example  \cite[Corollary 4.22]{AB}.

In the proposition on page 144 in \cite{AB}, Auslander and Bridger proved:
\begin{proposition}
Let $R$ be a two-sided noetherian ring. Then the following are equivalent:
\begin{enumerate}
    \item Each finitely generated left $R$-module which is 2-reflexive is reflexive.
    \item If for a finitely generated right $R$-module $N$ we have that $\Ext_R^2(N,R)$ is projective, then $\Ext_R^2(N,R)$ is zero.
\end{enumerate}
\end{proposition}
Auslander and Bridger raised the following question in the remark of page 145 in \cite{AB}:
\begin{question}
Let $R$ be a two-sided noetherian ring and $N$ a finitely generated $R$-module. Can $\Ext_R^2(N,R)$ be non-zero projective?
\end{question}
It seems that this question did not find an answer yet in the literature in general and by the previous proposition this question is equivalent to the question whether the notions of reflexive and 2-reflexive modules coincide. It was proven for commutative $R$ in the proposition on page 145 of \cite{AB} that $\Ext_R^2(N,R)$ can never be projective if it is non-zero and in for commutative rings 2-reflexive modules and the classical reflexive modules coincide.

The goal of this article is to give a first example of a finitely generated $R$-module $N$ with $\Ext_R^2(N,R)$ non-zero projective and a first example of a 2-reflexive module that is not reflexive. The example will even be a finite dimensional quiver algebra and thus in particular an artinian ring. We refer the reader for example to the textbook \cite{ASS} for basics on the representation theory of quiver algebras.
\begin{proposition} \label{mainresult}
    Let $Q$ be the quiver 
\[\begin{tikzcd}
	1 & 2 & 3 & 4
	\arrow["a", from=1-1, to=1-2]
	\arrow["d"', shift right=5, from=1-1, to=1-3]
	\arrow["b", from=1-2, to=1-3]
	\arrow["c", from=1-3, to=1-4]
\end{tikzcd}\]
and let $R=KQ/\rad^2(KQ)$ the radical square zero algebra with quiver $Q$. Denote by $S_i$ the simple $R$-modules.

\begin{enumerate}
    \item The indecomposable $R$-module $N$ with dimension vector $[1,0,1,0]$ satisfies that $\Ext_R^2(N,R)$ is a non-zero projective module.
    \item The simple $R$-module $S_3$ is 2-reflexive but not reflexive.
\end{enumerate}
\end{proposition}

\begin{proof}
Let $e_i$ denote the primitive idempotents corresponding to the vertices $i$ in the quiver of $R$. \newline
(1) Note that we have $\Omega^1(N) \cong S_2$ and
we have the exact sequence 
$$0 \rightarrow S_3 \rightarrow e_2 R \rightarrow S_2 \rightarrow 0$$
and applying the functor $\Hom_R(-,R)$ to this short exact sequence we obtain the exact sequence
$$0 \rightarrow \Hom_R(S_2,R) \rightarrow \Hom_R(e_2 R,R) \rightarrow \Hom_R(S_3,R) \rightarrow \Ext_R^1(S_2,R) \rightarrow 0.$$
Now the left $R$-module $\Hom_R(S_2,R)$ has dimension vector $[1,0,0,0]$, the left $R$-module $\Hom_R(e_2 R,R)$ has dimension vector $[1,1,0,0]$ and the left $R$-module $\Hom_R(S_3,R)$ has dimension vector $[1,1,0,0]$.
This implies that the left $R$-module $\Ext_R^2(N,R) \cong \Ext_R^1(S_2,R)$ has dimension vector $[1,0,0,0]$ and thus is a non-zero projective left $R$-module. \newline
(2) Let $M$ be a general $R$-module such that $M$ is 1-torsionfree (meaning it is a submodule of a projective module), then $M$ is 2-reflexive if and only if the $R$-module $\Ext_R^2(\Tr M,R)$ is projective by the proposition on page 143 in \cite{AB}.
Now the $R$-module $S_3$ is 1-torsionfree since it is a submodule of the indecomposable projective $R$-module $e_2 R$.
Now we calculate $\Ext_R^2(\Tr S_3,R)$ and show that this module is projective to conclude that $S_3$ is 2-reflexive.
A minimal projective resolution of $S_3$ is given by
$$0 \rightarrow e_4 R \rightarrow e_3 R \rightarrow S_3 \rightarrow 0$$ and applying the functor $\Hom_R(-,R)$ to this exact sequence gives 
$$0 \rightarrow \Hom_R(S_3,R) \rightarrow \Hom_R(e_3 R,R) \rightarrow \Hom_R(e_4 R,R) \rightarrow \Tr(S_3) \rightarrow 0.$$
Now $\Hom_R(S_3,R) $ has dimension vector $[1,1,0,0]$, $\Hom_R(e_3 R,R)$ has dimension vector $[1,1,1,0]$ and $\Hom_R(e_4 R,R) $ has dimension vector $[ 0, 0, 1, 1 ]$. Thus $\Tr(S_3)$ has dimension vector $[0,0,0,1]$ and $\Omega^1(\Tr(S_3))$ is the left simple $R$-module $Y$ with dimension vector $[0,0,1,0]$.
We have the exact sequence of left $R$-modules:
$$0 \rightarrow Y_2 \rightarrow R e_3 \rightarrow Y \rightarrow 0,$$
where $Y_2$ is the left $R$-module with dimension vector $[ 1, 1, 0, 0 ]$. Applying the functor $\Hom_R(-,R)$ to this exact sequence gives the exact sequence:
$$ 0 \rightarrow \Hom_R(Y,R) \rightarrow \Hom_R(R e_3 ,R) \rightarrow \Hom_R(Y_2,R) \rightarrow \Ext_R^1(Y,R) \rightarrow 0.$$
Now $\Hom_R(Y,R)$ has dimension vector $[0,0,0,1]$, $\Hom_R(Re_3,R)$ has dimension vector $[ 0, 0, 1, 1 ]$ and $\Hom_R(Y_2,R)$ has dimension vector $[1,1,2,0]$.
Thus the $R$-module $\Ext_R^1(Y,R)$ has dimension vector $[ 1, 1, 1, 0 ]$.
Now the module $\Hom_R(Y_2,R)$ with dimension vector $[1,1,2,0]$ is isomorphic to the $R$-module $S_3 \oplus e_1 R$ and $\Ext_R^1(Y,R)$ being a quotient of $\Hom_R(Y_2,R) \cong S_3 \oplus e_1 R$ with dimension vector $[1,1,1,0]$ implies that $\Ext_R^1(Y,R) \cong e_1 R$ is projective. Thus $\Ext_R^2(\Tr (S_3),R) \cong \Ext_R^1(Y,R)$ is projective and $S_3$ is 2-reflexive.
What is left to show is that $S_3$ is not reflexive. But $S_3$ being reflexive is equivalent to $\Ext_R^i(\Tr (S_3),R)=0$ for $i=1,2$, see for example \cite[corollary 3.3 in chapter IV.]{ARS}.
But we just saw that $\Ext_R^2(\Tr (S_3),R)$ is non-zero and thus $S_3$ can not be reflexive.

\end{proof}

\section{Acknowledgement}
The algebra and the modules $N$ and $S_3$ in Proposition \ref{mainresult} were found and verified using the GAP-package \cite{QPA}.

\end{document}